\documentclass{amsart}
\usepackage{amssymb,amsmath, amsfonts, amsrefs, tikz, epsfig, float}
\usepackage{amsthm}
\usepackage{mathrsfs}
\usepackage{enumitem, indentfirst}
\usepackage[fleqn,tbtags]{mathtools}
\usepackage{color}
\newtheorem{Thm}{Theorem}[section]
\newtheorem{Lem}[Thm]{Lemma}
\newtheorem{Prop}[Thm]{Proposition}

\newtheorem{Cor}[Thm]{Corollary}

\theoremstyle{definition}

\newtheorem{Ex}[Thm]{Example}
\newtheorem{Rem}[Thm]{Remark}

\newtheorem*{Ack}{Acknowledgement}

\numberwithin{equation}{section}
\renewcommand{\epsilon}{\varepsilon}
\newcommand{\bh}{\mc B(\mc H)}
\DeclareMathOperator*{\Ran}{\mathrm{Ran}}
\DeclareMathOperator*{\Ker}{\mathrm{Ker}}

\DeclareMathOperator*{\wlim}{\mathrm{w-lim}}
\DeclareMathOperator*{\wslim}{\mathrm{w*-lim}}

\DeclareMathOperator*{\spanning}{\mathrm{span}}

\DeclareMathOperator*{\conv}{\mathrm{conv}}

\DeclareMathOperator*{\density}{\mathrm{dens}}

\newcommand{\mc}{\mathcal}

\title[Perturbations of epimorphisms]{Perturbations of surjective homomorphisms between algebras of operators on Banach spaces}
\author{Bence Horv\'{a}th and Zsigmond Tarcsay}

\newcommand{\Addresses}{{
  \bigskip
  \footnotesize
  
  \textit{Addresses}: \\
  Bence Horv\'{a}th: \textsc{Institute of Mathematics, Czech Academy of Sciences, \v{Z}itn\'a 25, 115 67 Prague 1, Czech Republic} \\
  Zsigmond Tarcsay: \textsc{
Alfr\'ed R\'enyi Institute of Mathematics, Re\'altanoda utca 13-15., Budapest H-1053,  Hungary and Department of Applied Analysis \\
Computational Mathematics, E\"otv\"os Lor\'and University, P\'azm\'any P\'eter  s\'et\'any 1/c., Budapest H-1117, Hungary}\par\nopagebreak
  \textit{E-mail addresses}: \\
  Bence Horv\'{a}th: \texttt{horvath@math.cas.cz, hotvath@gmail.com} \\
  Zsigmond Tarcsay: \texttt{zsigmond.tarcsay@ttk.elte.hu}
}}
\thanks{The corresponding author B. Horváth acknowledges with thanks the funding received from GA\v{C}R project 19-07129Y; RVO 67985840 (Czech Republic)}
\thanks{Zs. Tarcsay was supported by DAAD-TEMPUS Cooperation Project ``Harmonic Analysis and Extremal Problems'' (grant no. 308015),  by the J\'anos Bolyai Research Scholarship of the Hungarian Academy of Sciences, and by the \'UNKP--20-5-ELTE-185 New National Excellence Program of the Ministry for Innovation and Technology. ``Application Domain Specific Highly Reliable IT Solutions'' project  has been implemented with the support provided from the National Research, Development and Innovation Fund of Hungary, financed under the Thematic Excellence Programme TKP2020-NKA-06 (National Challenges Subprogramme) funding scheme.}
\begin{document}	

\begin{abstract}
A remarkable result of Moln\' ar [\textit{Proc. Amer. Math. Soc.},~126~(1998), 853--861] states that automorphisms of the algebra of operators acting on a separable Hilbert space are stable under ``small'' perturbations. More precisely, if $\phi,\psi$ are endomorphisms of $\bh$ such that $\|\phi(A)-\psi(A)\|<\|A\|$ and $\psi$ is surjective, then so is $\phi$. The aim of this paper is to extend this result to a larger class of Banach spaces including $\ell_p$  and $L_p$ spaces, where $1<p<\infty$.

\textit{En route} to the proof we show that for any Banach space $X$ from the above class all faithful, unital, separable, reflexive representations of $\mc B (X)$ which preserve rank one operators are in fact isomorphisms.
\end{abstract}

\maketitle

\let\thefootnote\relax\footnote{\subjclass{2020}{\textit{ Mathematics Subject Classification.} Primary 46H10, 47L10; Secondary 46B03, 46B07, 46B10, 47L20}

}

\section{Introduction and statement of main results}

It is a well known corollary of the Carl Neumann series that if $\psi$ and $\phi$ are endomorphisms of a Banach algebra, and $\psi$ is an automorphism with $\| \psi - \phi \| < 1/ \| \psi^{-1} \|$, then $\phi$ is an automorphism too. Motivated by this fact, Moln\'ar proved in \cite{molnar2} that in fact a sharper version of this result holds for $\mc B (\mc H)$, the $C^*$-algebra of bounded linear operators on a separable Hilbert space $\mc H$. More precisely, he showed in \cite{molnar2}*{Theorem~1} that if $\phi, \psi \colon \mathcal{B}(\mc H) \rightarrow \mathcal{B}(\mathcal{H})$ are algebra homomorphisms such that $\psi$ is surjective and $\| \psi(A) - \phi(A) \| < \|A\|$ for each non-zero $A \in \mathcal{B}(\mathcal{H})$, then $\phi$ is also surjective. Let us remark here that $\psi$ and $\phi$ are automatically continuous, and their surjectivity implies their injectivity, as shown in the proof of \cite{molnar2}*{Theorem~1}, for example. The main tool in Moln\'ar's proof is a previous, deep result of his from \cite{molnar1}.

The purpose of this paper is to extend \cite{molnar2}*{Theorem~1}  for a large class of (non-hilbertian) Banach spaces; see Theorem \ref{stability}. \textit{En route} to this we shall prove a theorem about certain faithful representations of $\mc B (X)$, where $X$ is a Banach space from the same class; see Theorem \ref{main}. We believe this result to be of independent interest, since the study of faithful, separable representations of $\mc B (X)$ goes back to the seminal work of Berkson and Porta in \cite{berksonporta}. Our main results are the following:

\begin{Thm}\label{main}
	Let $X$ and $Y$ be non-zero Banach spaces such that $Y$ is separable and reflexive. Assume $X$ satisfies one of the following:
	\begin{enumerate}[label=\textup{(\arabic*)}]
	\item $X=L_p[0,1]$, where $1<p<\infty$; or
	\item $X$ is a reflexive Banach space with a subsymmetric Schauder basis.
	\end{enumerate}
	Let $\phi \colon \mathcal{B}(X) \rightarrow \mathcal{B}(Y)$ be a continuous, injective algebra homomorphism. If $\Ran(\phi)$ contains an operator with dense range, and $\phi$ maps rank one idempotents into rank one idempotents, then $\phi$ is an isomorphism.	
\end{Thm}

From the above theorem we will deduce a generalisation of \cite{molnar2}*{Theorem~1} :

\begin{Thm}\label{stability}
	Let $X$ and $Y$ be Banach spaces as in Theorem \ref{main}. Let $\psi, \phi \colon \mc B (X) \rightarrow \mc B (Y)$ be algebra homomorphisms such that $\psi$ is surjective. If
	\begin{align}\label{mainineq}
	\| \psi(A) - \phi(A) \| < \| A \|
	\end{align}
	for each non-zero $A \in \mc B (X)$, then $\phi$ is an isomorphism.
\end{Thm}

As one might expect, there is no hope for Theorem~\ref{main} to hold in general for arbitrary Banach spaces $X$ and $Y$. To be precise, we prove the following:

\begin{Prop}\label{counterexample}
Let $X$ be a Banach space such that $\mc B (X)$ has a character. Let $Z$ be any non-zero Banach space. There is a continuous, injective algebra homomorphism $\phi \colon \mc B (X) \rightarrow \mc B(X \oplus Z)$ with $\phi(I_{X}) = I_{X \oplus Z}$ which maps rank one operators into rank one operators but $\phi$ is not surjective.

In particular, let $X$ be the $p^{th}$ James space $\mathcal{J}_p$ (where $1<p< \infty$) or the Semadeni space $C[0, \omega_1]$. There is a continuous, injective algebra homomorphism $\phi \colon \mc B (X) \rightarrow \mc B(X)$ with $\phi(I_{X}) = I_{X}$ which maps rank one operators into rank one operators but $\phi$ is not surjective.
\end{Prop}

The necessary terminology will be explained in the subsequent sections.

The paper is structured as follows. Section~2 contains a brief overview of the concepts and notation needed to understand the paper. In Section~3 we develop some auxiliary tools which will feature heavily in our arguments later. Section~4 is devoted to the proofs of Theorems~\ref{main},~\ref{stability} and Proposition~\ref{counterexample}. We conclude Section~4 with some remarks about the possibility of weakening the assumptions in Theorem~\ref{main}.

\section{Preliminaries}\label{sec:prelin}

The notation and terminology used throughout this paper are standard.  

\subsubsection{Numbers and sets}

The first infinite cardinal is denoted by $\aleph_0$ and we refer to the cardinal $2^{\aleph_0}$ as the \textit{continuum}. If $X$ is a set, then $\mathcal{P}(X)$ denotes its power set, and $\vert X \vert$ denotes the cardinality of $X$. If $X$ and $Y$ are sets, then $Y^X$ is the set of functions from $X$ to $Y$. 

Let $\Gamma$ be a set. A family $\mathcal{F} \subseteq \mathcal{P}(\Gamma)$ is called \textit{almost disjoint} if for any distinct $A,B \in \mathcal{F}$ the set $A \cap B$ is finite. There exists an almost disjoint family of continuum cardinality consisting of infinite subsets of the natural numbers. For a proof we refer the reader to \emph{e.g.} \cite{albiackalton}*{Lemma~2.5.3}.

\subsubsection{Ultrafilters, ultralimits}

If $\mathcal{F}$ is a filter on a set $X$ and $\mathcal{U}$ is an ultrafilter on $X$ with $\mathcal{F} \subseteq \mathcal{U}$, then we say that \textit{$\mathcal{U}$ extends $\mathcal{F}$}. As a corollary of Zorn's Lemma any filter can be extended to an ultrafilter.

Let $X$ be a topological space and let $x \in X$. Let $(x_i)_{i \in I}$ be a family of elements of $X$ and let $\mathcal{U}$ be an ultrafilter on $I$. If $(x_i)_{i \in I}$ converges to $x$ along $\mathcal{U}$, then we will denote this by $x= \lim_{i \rightarrow \mathcal{U}} x_i$.
Let $X$ be a compact Hausdorff space, and let $(x_i)_{i \in I}$ be a family of elements of $X$. If $\mathcal{U}$ is an ultrafilter on $I$, then the ultralimit $\lim_{i \rightarrow \mathcal{U}} x_i$ exists and it is unique (see \emph{e.g.} \cite{Alip}*{Lemma~1.5.9}). If $X$ and $Y$ are topological spaces and $f \colon X\to Y$ is a continuous function, then $\lim_{i\to\mc U}x_i=x$ implies $\lim_{i\to\mc U} f(x_i)=f(x)$.  

There is a standard way of connecting convergence in a topological space with convergence along certain ultrafilters. Let $I$ be a directed set. We define $A_i= \lbrace j \in I \colon j \geq i \rbrace$ for every $i \in I$. Then the set
\begin{align*}
	\mathcal{F}_{\text{ord}}= \lbrace S \in \mathcal{P}(I) \colon A_i \subseteq S \text{ for some } i \in I \rbrace
\end{align*}
is easily seen to be a filter on $I$, called the \textit{order filter}.	

Let $X$ be a topological space and let $(x_i)_{i \in I}$ be a family of elements of $X$ which converges to $x \in X$. If $\mathcal{U}$ is an ultrafilter on $I$ with $\mathcal{F}_{\text{ord}} \subseteq \mathcal{U}$, then $x= \lim_{i \rightarrow \mathcal{U}} x_i$.

\subsection{Background material on Banach spaces and Banach algebras}

In this paper all Banach spaces and Banach algebras are assumed to be complex.

\subsubsection{The dual space and the $\text{weak}^*$ topology}

If $X$ is a Banach space, then for its dual space we write $X^*$. In the following $\langle \cdot , \cdot \rangle$ denotes the natural duality pairing; that is, $\langle x, f \rangle = f(x)$ whenever $x \in X$ and $f \in X^*$. The $\text{weak}^*$ topology on $X^*$ is denoted by $\sigma(X^*,X)$.

\subsubsection{Operators on Banach spaces}

The identity operator on a vector space $X$ is denoted by $I_X$. If $X$ and $Y$ are normed spaces, then $\mathcal{B}(X,Y)$ denotes the normed space of bounded linear operators from $X$ to $Y$. We denote $\mathcal{B}(X,X)$ simply by $\mathcal{B}(X)$. For $T \in \mathcal{B}(X,Y)$ its adjoint is denoted by $T^*$. If $Z$ is a linear subspace of $X$, then for $T \in \mathcal{B}(X,Y)$ we denote the restriction of $T$ to $Z$ by $T \vert_Z$; clearly $T \vert_Z \in \mathcal{B}(Z,Y)$.

Let $X$ and $Y$ be normed spaces. Let $y \in Y$ and let $f \in X^*$. We define $y \otimes f$ by $(y \otimes f)(x) = \langle x, f \rangle y$. It is clear that $y \otimes f \in \mathcal{B}(X,Y)$ is rank one with $\Vert y \otimes f \Vert = \Vert y \Vert \Vert f \Vert$, whenever $y \in Y$ and $f \in X^*$ are non-zero.

Two Banach spaces $X$ and $Y$ are said to be \textit{isomorphic} if there is a linear homeomorphism between $X$ and $Y$; it will be denoted by $X \simeq Y$.

\subsubsection{Banach algebras, idempotents}

By an \textit{isomorphism of Banach algebras} $\mc A$ and $\mc B$ we understand that there is an algebra homomorphism between $\mc A$ and $\mc B$ which is also a homeomorphism. This will also be denoted by $\mc A \simeq \mc B$.

In an algebra $\mc A$ an element $p \in \mc A$ is an \emph{idempotent} if $p^2 =p$. Two idempotents $p,q \in \mc A$ are \textit{orthogonal} if $qp=0=pq$. We say that two idempotents $p,q \in \mc A$ are \emph{algebraically Murray--von Neumann equivalent} or simply \emph{equivalent}, and denote it by $p \sim q$, if there exist $a,b \in \mc A$ such that $ab=p$ and $ba=q$. For idempotents $p,q \in \mc A$ we write $p \leq q$ whenever $pq=p=qp$. Clearly $\leq$ is a partial ordering on the set of idempotents of $\mc A$.

We recall a folklore result, a stronger version of which was proved by Zem\'{a}nek in \cite{zemanek0}*{Lemma~3.1}. A self-contained elementary proof can be found in \cite{dawshorvath}*{Lemma 2.8}.

\begin{Lem}\label{zemanek}
	Let $\mathcal{A}$ be a unital Banach algebra. Let $p,q \in \mathcal{A}$ be idempotents with $\Vert p -q \Vert <1$. Then $p \sim q$.
\end{Lem}

Let $X$ be a Banach space. Two idempotents $P,Q \in \mathcal{B}(X)$ are said to be \textit{almost orthogonal} if $PQ$ and $QP$ are finite-rank operators.

The following lemma is well-known; see for example \cite{ringfinofopalgs}*{Lemma~1.4}.
\begin{Lem}\label{equivisom}
	Let $X_1$ and $X_2$ be Banach spaces. Let $P \in \mathcal{B}(X_1)$ and $Q \in \mathcal{B}(X_2)$ be idempotents. Then $\Ran(P) \simeq \Ran(Q)$ as Banach spaces if and only if there exist $U \in \mathcal{B}(X_2,X_1)$ and $V \in \mathcal{B}(X_1,X_2)$ with $P = U   V$ and $Q = V   U$. In particular, when $X_1=X_2$ we have $P \sim Q$ if and only if $\Ran(P) \simeq \Ran(Q)$.
\end{Lem}

\subsubsection{Ideals of operators on Banach spaces}

Let $X$ and $Y$ be Banach spaces. Let $T \in \mathcal{B}(X,Y)$. Then $T$ is a \textit{finite-rank operator} if $\Ran(T)$ is finite-dimensional. The symbol $\mathcal{F}(X,Y)$ stands for the set of finite-rank operators on $X$. It is well known that $\mathcal{F}$ is the smallest operator ideal in the sense of Pietsch; see for example \cite{pietsch}*{Theorem~1.2.2}. In an infinite-dimensional Banach space $X$, $\mathcal{F}(X)$ is always a proper, non-closed, two-sided ideal.

The symbol $\mathcal{A}(X,Y)$ stands for the (operator) norm-closure of $\mathcal{F}(X,Y)$. It is clear that $\mathcal{A}(X)$ is the smallest closed, proper, non-zero, two-sided ideal in $\mc B(X)$. An element of $\mathcal{A}(X,Y)$ is called an \textit{approximable operator}. The set of compact operators from $X$ to $Y$ is denoted by $\mathcal{K}(X,Y)$. It is known that $\mathcal{K}$ is a closed operator ideal such that $\mathcal{A} \subseteq \mathcal{K}$.

\subsubsection{Schauder bases in Banach spaces}\label{bases}

Let $X$ be a Banach space with Schauder basis $(b_n)_{n \in \mathbb{N}}$. Then $f_n \in X^*$ and $P_n \in \mathcal{B}(X)$ denote the corresponding $n^{\text{th}}$ coordinate functional and projection, respectively, for all $n \in \mathbb{N}$. It is standard that $(P_n)_{n \in \mathbb{N}}$ converges to $I_X$ in the strong operator topology. In particular, $(P_n)_{n \in \mathbb{N}}$ is uniformly bounded by the Banach--Steinhaus Theorem. We remark in passing that if a Banach space $X$ has a basis, then $\mathcal{A}(X) = \mathcal{K}(X)$.

Recall that if $(b_n)_{n \in \mathbb{N}}$ is an unconditional basis in $X$, then for any $A \subseteq \mathbb{N}$ 
\begin{align}\label{equividemp}
		P_A \colon X \rightarrow X; \quad x \mapsto \sum\limits_{n \in A} \langle x, f_n \rangle b_n
\end{align}
defines a bounded linear idempotent operator on $X$ and the family $(P_A)_{A\in\mathcal{P}(\mathbb N)}$ is uniformly bounded. A basis $(b_n)_{n \in \mathbb{N}}$ of $X$ is called \textit{subsymmetric} if\index{subsymmetric basis} it is an unconditional basis and the basic sequence $(b_{\sigma(n)})_{n \in \mathbb{N}}$ is equivalent to $(b_n)_{n \in \mathbb{N}}$ for every strictly monotone increasing function $\sigma \colon \mathbb{N} \rightarrow \mathbb{N}$. We note that the natural bases for $c_0$ and $\ell_p$ (where $1 \leq p < \infty$) are subsymmetric; see \cite{albiackalton}*{Section~9.2}. For $p \in [1, \infty) \backslash \lbrace 2 \rbrace$ the space $L_p[0,1]$ does not have a subsymmetric basis; see \cite{Singer1}*{Chapter~II,~Theorem~21.1~on~p.~568}. In fact, $L_1[0,1]$ does not even have an unconditional basis by \cite{albiackalton}*{Theorem~6.3.3}.

The following well known fact can be found, for example, in the monograph of Lindenstrauss and Tzafriri. We refer to the paragraph after \cite{LT}*{Definition~3.a.2}.

\begin{Prop}\label{subsymnorm}
	Let $X$ be a Banach space with a subsymmetric basis $(b_n)_{n \in \mathbb{N}}$. For any strictly monotone increasing function $\sigma \colon \mathbb{N} \rightarrow \mathbb{N}$ the map 
	\begin{align}\label{shiftop}
		S_{\sigma} \colon X \rightarrow X; \quad x \mapsto \sum\limits_{n \in \mathbb{N}} \langle x, f_n \rangle b_{\sigma(n)}
	\end{align}
	is an isomorphism onto its range.
\end{Prop}

We recall that a Schauder basis $(b_n)_{n \in \mathbb{N}}$ for a Banach space $X$ is \textit{shrinking} if the sequence of coordinate functionals $(f_n)_{n \in \mathbb{N}}$ associated with $(b_n)_{n \in \mathbb{N}}$ is a Schauder basis for $X^*$. Any Schauder basis in a reflexive Banach space is shrinking (see \cite{Singer1}*{Chapter~II, Example~4.3 on p.~278}). Clearly $\ell_1$ and $L_1[0,1]$ cannot have shrinking bases since their dual spaces are non-separable.

\subsection{Dual Banach algebras and approximate identities}

For our definition of a dual Banach algebra we follow \cite{dawsdual}*{Sections~1~and~2} and \cite{Runde1}*{Chapter~5}. Let $\mc B$ be a Banach algebra and let $X$ be a Banach space such that $\mc B$ and $X^*$ are isometrically isomorphic as Banach spaces. We say that $\mc B$ is a \textit{dual Banach algebra} with \textit{predual} $X$ if multiplication on $\mc B$ (henceforth identified with $X^*$ as a Banach space) is separately $\sigma(X^*,X)$ - to - $\sigma(X^*,X)$ continuous. We remark in passing that this latter condition is equivalent to saying that the image $\Ran(\kappa_{X})$ of the canonical embedding $\kappa_{X} \colon X \to X^{**}$ is a sub-$\mc B$-bimodule of $\mc B^*$. See \cite{dawsdual}*{Lemma~2.2}.

If $X$ is a Banach space, then the projective tensor product of $X$ and $X^*$ is denoted by $X \widehat{\otimes}_{\pi} X^*$. For background information on the projective tensor products of Banach spaces we refer the reader to \cite{defflor} and \cite{Ryan}.

The following result is taken from \cite{Runde1}*{Example~5.1.4}:
\begin{Lem}\label{bxisdualbanalg}
	Let $X$ be a reflexive Banach space. Then there is an isometric isomorphism between $\mathcal{B}(X)$ and $(X \widehat{\otimes}_{\pi} X^*)^*$ given by
	\begin{align*}
		\langle x \otimes f, A \rangle = \langle Ax, f \rangle
	\end{align*}
	for any $x \in X, f \in X^*$ and $A \in \mathcal{B}(X)$. Moreover, $\mathcal{B}(X)$ is a dual Banach algebra with predual $X \widehat{\otimes}_{\pi} X^*$.
\end{Lem}

Let $\mc A$ be a Banach algebra. A net $(e_{\gamma})_{\gamma \in \Gamma}$ in $\mc A$ is a \textit{bounded left (respectively, right) approximate identity} if $\sup_{\gamma} \Vert e_{\gamma} \Vert < \infty$ and $\lim_{\gamma} e_{\gamma} a = a$  (respectively, $ \lim_{\gamma} ae_{\gamma} = a $) for every $a \in \mc A$. A net $(e_{\gamma})_{\gamma \in \Gamma}$ is a \textit{bounded approximate identity (b.a.i.)} if it is a bounded left and right approximate identity.	

The following is an immediate consequence of \cite{Dales}*{Theorem~2.9.37}.

\begin{Cor}\label{aiai2}
	Let $X$ be a Banach space with a Schauder basis. Then the sequence of coordinate projections $(P_n)_{n \in \mathbb{N}}$ is a bounded left approximate identity for $\mathcal{K}(X)$. If $X$ has a shrinking basis, then  $(P_n)_{n \in \mathbb{N}}$ is a bounded approximate identity for $\mathcal{K}(X)$.
\end{Cor}			

\section{Some auxiliary results}

Let $X$ be a Banach space. Let $(f_i)_{i \in I}$ be a family of elements of the topological space $(X^*, \sigma(X^*,X))$. Let $\mathcal{U}$ be an ultrafilter on $I$ such that the ultralimit of $(f_i)_{i \in I}$ along $\mathcal{U}$ with respect to the topology $\sigma(X^*,X)$ exists in $X^*$. This limit will be denoted by $\wslim_{i \rightarrow \mathcal{U}} f_i$.

\begin{Lem}\label{idemplemmaster}
	Let $\mc B$ be a dual Banach algebra. Let $(q_{\gamma})_{\gamma \in \Gamma}$ be a bounded net in $\mc B$ such that $\lim_{\omega}q_{\omega} q_{\gamma} = q_{\gamma}$ in norm for any $\gamma \in \Gamma$. 
	Then $p=\wslim_{\gamma \rightarrow \mathcal{U}} q_{\gamma} \in \mc B$ exists and defines an idempotent whenever $\mathcal{U}$ is an ultrafilter on $\Gamma$ which extends the order filter.
\end{Lem}

\begin{proof}
	Let $\mathcal{U}$ be an ultrafilter on $\Gamma$ which extends the order filter. By the Banach--Alaoglu Theorem $p=\wslim_{\gamma \rightarrow \mathcal{U}} q_{\gamma} \in \mc B$ is well-defined. We show that $p \in \mc B$ is idempotent. As multiplication on $\mc B$ is separately $\text{weak}^*$ continuous, we have
	\begin{align*}
	pq_{\gamma}&= \left(\wslim\limits_{\omega \rightarrow \mathcal{U}} q_{\omega} \right)q_{\gamma} = \wslim\limits_{\omega \rightarrow \mathcal{U}} q_{\omega} q_{\gamma} = q_{\gamma} \quad (\text{for all } \gamma \in \Gamma),
	\end{align*}
	because $\lim_{\omega \rightarrow \mathcal{U}} q_{\omega} q_{\gamma} = q_{\gamma}$ in norm. Consequently,
	\begin{align*}
	p^2 &= p \left(\wslim\limits_{\gamma \rightarrow \mathcal{U}} q_{\gamma} \right) = \wslim\limits_{\gamma \rightarrow \mathcal{U}} p q_{\gamma} = \wslim\limits_{\gamma \rightarrow \mathcal{U}} q_{\gamma} = p.
	\end{align*}
	This shows that $p$ is an idempotent as claimed.
\end{proof}

The following lemma has many ``folklore'' variations (see \textit{e.g.} \cite{berksonporta}*{Lemma~2.23}). Rather than hunt for a reference which states Lemma~\ref{projsot} exactly in the form suitable for our purpose, we shall prove the result here.

\begin{Lem}\label{projsot}
Let $X$ be a reflexive Banach space. Let $(Q_n)_{n \in \mathbb{N}}$ be a bounded, monotone increasing sequence of idempotents in $\mc B (X)$. Then there exists an idempotent $Q \in \mc B (X)$ such that $(Q_n)_{n \in \mathbb{N}}$ converges to $Q$ in the strong operator topology.
\end{Lem}
	
\begin{proof}
  Let $\mathcal{U}$ be a free ultrafilter on $\mathbb{N}$. We show first that there exists an idempotent $Q \in \mathcal{B}(X)$ such that $(Q_n)$ converges to $Q$ along $\mathcal{U}$ in the weak operator topology. Let $X \widehat{\otimes}_{\pi} X^*$ be the predual of $\mathcal{B}(X)$ as in Lemma~\ref{bxisdualbanalg}. According to Lemma~\ref{idemplemmaster},  $Q= \wslim_{n \rightarrow \mathcal{U}} Q_n$ is a well-defined idempotent operator in $\mathcal{B}(X)$. It remains to show that $\lim_{n \rightarrow \mathcal{U}} \langle Q_n x, f \rangle = \langle Qx, f \rangle$ for any $x \in X$ and $f \in X^*$. This is a simple calculation:
 \begin{align*}
 \langle Qx,f \rangle &= \langle x \otimes f, Q \rangle \notag = \lim\limits_{n \rightarrow \mathcal{U}} \langle x \otimes f, Q_n \rangle = \lim\limits_{n \rightarrow \mathcal{U}} \langle Q_n x, f \rangle.
 \end{align*}
We show that $(Q_n)$ converges to $Q$ in the strong operator topology. Firstly let us observe that $Q_n Q = Q_n$ for any $n \in \mathbb{N}$. Indeed, for any $z \in X$ and $f \in X^*$ we have 
\begin{align*}
\langle Q_n Qz,f \rangle & = \langle Qz, Q_n^* f \rangle = \lim\limits_{i \rightarrow \mathcal{U}} \langle Q_i z, Q_n^*f \rangle = \lim\limits_{i \rightarrow \mathcal{U}} \langle Q_n Q_i z, f \rangle = \langle Q_n z, f \rangle, 
\end{align*}
thus proving $Q_n Q = Q_n$. A similar argument shows $QQ_n = Q_n$ for all $n \in \mathbb{N}$. Let us now fix $x \in X$. Clearly $Q_n x \in \conv \lbrace Q_m x \colon m \in \mathbb{N} \rbrace$ for any $n \in \mathbb{N}$. Therefore $Qx= \wlim_{n \rightarrow \mathcal{U}} Q_n x$ with Mazur's Theorem (see \emph{e.g.} \cite{Megginson}*{Theorem~2.5.16}) implies that $Qx \in \overline{\conv} \lbrace Q_m x \colon m \in \mathbb{N} \rbrace$, where the closure is taken with respect to the norm topology of $X$. Let us fix $\epsilon >0$. There exist a finite set $\Gamma \subseteq \mathbb{N}$ and $(\lambda_j)_{j \in \Gamma}$ in $[0,1]$ such that $\textstyle{\| Qx - \sum_{j \in \Gamma} \lambda_j Q_j x \|} < \epsilon/(K+1)$. Let $N= \max(\Gamma)$, then $\textstyle{ Q_n ( \sum_{j \in \Gamma} \lambda_j Q_j ) = \sum_{j \in \Gamma} \lambda_j Q_j}$ for any $n \geq N$. Consequently for each $n \geq N$:
\begin{align*}
\Vert Qx - Q_n x \Vert & \leq \Big\| Qx - \sum\limits_{j \in \Gamma} \lambda_j Q_j x \Big\| + \Big\| Q_n \Big( \sum\limits_{j \in \Gamma} \lambda_j Q_j x - Qx \Big) \Big\| \notag \\
& < \dfrac{\epsilon}{K+1} + K \dfrac{\epsilon}{K+1} = \epsilon.
\end{align*}
This shows that $(Q_n x)$ converges to $Qx$ in $X$ as required.
\end{proof}	

\begin{Lem}\label{techlem1}
	Let $\mc A$ be a Banach algebra. Let $\mc J \trianglelefteq \mc A$ be a closed, two-sided ideal with a b.a.i. $(e_{\gamma})_{\gamma \in \Gamma}$. Let $\mc B$ be a unital, dual Banach algebra. Suppose $\psi \colon \mc A \rightarrow \mc B$ is a continuous algebra homomorphism. If $\mathcal{U}$ is an ultrafilter on $\Gamma$ which extends the order filter, then:
	\begin{enumerate}[label=\textup{(\arabic*)}]
		\item $p= \wslim_{\gamma \rightarrow \mathcal{U}} \psi(e_{\gamma}) \in \mc B$ is an idempotent;
		\item $p \psi(c) = \psi(c) = \psi(c) p$ for all $c \in \mc J$;
		\item $p \psi(a) = p \psi(a) p = \psi(a) p$ for all $a \in \mc A$;
		\item the map
		\begin{align}
	\theta \colon \mc A &\rightarrow \mc B; \quad a \mapsto (1_{\mc B}-p) \psi(a) (1_{\mc B}-p)
	\end{align}
	is a continuous algebra homomorphism with $\theta |_{\mc J} = 0$.
	\end{enumerate}
\end{Lem}

\begin{proof}
	(1) Since we have $\lim_{\gamma} \psi(e_{\gamma}) \psi(e_{\omega}) = \lim_{\gamma} \psi(e_{\gamma} e_{\omega}) = \psi(e_{\omega})$ and similarly,  $\lim_{\gamma} \psi(e_{\omega}) \psi(e_{\gamma}) = \lim_{\gamma} \psi(e_{\omega} e_{\gamma}) = \psi(e_{\omega})$ for any $\omega \in \Gamma$,  the statement follows from Lemma~\ref{idemplemmaster}.
	
Before we proceed we observe that for any $a \in \mc A$
\begin{align}\label{unifcalc1}
	p \psi(a) &= \left( \wslim\limits_{\gamma \rightarrow \mathcal{U}} \psi(e_{\gamma}) \right) \psi(a) = \wslim\limits_{\gamma \rightarrow \mathcal{U}} \psi(e_{\gamma}) \psi(a) = \wslim\limits_{\gamma \rightarrow \mathcal{U}} \psi(e_{\gamma} a).
\end{align}

(2) Let us fix $c \in \mc J$. Then from \eqref{unifcalc1} and the fact that $(e_{\gamma})$ is a b.a.i. for $\mc J$ we obtain $p \psi(c) = \wslim_{\gamma \rightarrow \mathcal{U}} \psi(e_{\gamma} c) = \psi(c)$. An analogous argument shows $\psi(c)p = \psi(c)$.

(3) Let us fix $a \in \mc A$. Since $e_{\gamma} a \in \mc J$ for any $\gamma \in \Gamma$, it follows from (2) that $\psi(e_{\gamma} a) = \psi(e_{\gamma} a) p = \psi(e_{\gamma}) \psi(a) p$. From this and \eqref{unifcalc1} we obtain
\begin{align*}
p \psi(a) &= \wslim\limits_{\gamma \rightarrow \mathcal{U}} \psi(e_{\gamma} a) = \wslim\limits_{\gamma \rightarrow \mathcal{U}} \psi(e_{\gamma}) \psi(a) p  = \left(\wslim\limits_{\gamma \rightarrow \mathcal{U}} \psi(e_{\gamma}) \right) \psi(a) p =p \psi(a) p.
\end{align*}
A similar argument shows $\psi(a)p = p \psi(a) p$.

(4) It is clear that $\theta$ is a bounded linear map. Let us first fix $a \in \mathcal{A}$. From (3) we have $\psi(a)p = p \psi(a)p$ and hence
\begin{align}\label{denserange}
\theta(a) &= \psi(a) - \psi(a)p - p \psi(a) +p \psi(a)p = \psi(a) - p \psi(a).
\end{align}
From the above, another application of (3), and the fact that $\psi$ is an algebra homomorphism it follows that $\theta$ is multiplicative. Finally, it is straightforward from (2) that  $\theta |_{\mc J} =0$.	
\end{proof}

Before we proceed let us recall some basic probability-theoretic background and terminology. In the brief exposition below we follow Fremlin's book \cite{Fremlin}*{Sections~254J--254R}.
\begin{Rem}\label{condexp}
	We consider the the probability space $(\lbrace 0,1 \rbrace, \mathcal{P}(\lbrace 0,1 \rbrace), \mu)$ where $\mu$ is the ``fair coin'' probability measure, i.e.,  $\mu(\lbrace 0 \rbrace) = 1/2 = \mu(\lbrace 1 \rbrace)$. Let $(\lbrace 0,1 \rbrace^{\mathbb{N}}, \Lambda, \nu)$ denote the product of the family $\big((\lbrace 0,1 \rbrace, \mathcal{P}(\lbrace 0,1 \rbrace), \mu) \big)_{n \in \mathbb{N}}$ of probability spaces. The measure space $(\lbrace 0,1 \rbrace^{\mathbb{N}}, \Lambda, \nu)$ is isomorphic to $([0,1], \mathcal{A}, \lambda)$, where $\lambda$ is the Lebesgue measure restricted to $[0,1]$. 
	Consequently,  for all $p$ such that   $1 \leq p < \infty$, the Banach spaces $L_p(\lbrace 0,1 \rbrace^{\mathbb{N}}, \Lambda, \nu)$ and $L_p([0,1], \mathcal{A}, \lambda)$ are isometrically isomorphic (see also \cite{albiackalton}*{page~125}).
	
	For any $S \subseteq \mathbb{N}$ let us define 
	\begin{align*}
	\pi_{S} \colon \lbrace 0, 1 \rbrace^{\mathbb{N}} \rightarrow \lbrace 0, 1 \rbrace^{S}; \quad (x_n)_{n \in \mathbb{N}} \mapsto (x_n)_{n \in S}
	\end{align*}
	and
	\begin{align*}
	\Lambda_{S}= \Big\lbrace A \in \Lambda \colon A = \pi_{S}^{-1}[\pi_{S}[A]] \Big\rbrace.
	\end{align*}
	The set $\Lambda_S$ is a $\sigma$-subalgebra of $\Lambda$. In the case when $S$ is an infinite subset of $\mathbb{N}$, it follows that $(\lbrace 0,1 \rbrace^{\mathbb{N}}, \Lambda_S, \nu \vert_{\Lambda_S})$ is isomorphic to $([0,1], \mathcal{A}, \lambda)$. Thus $L_p(\lbrace 0,1 \rbrace^{\mathbb{N}}, \Lambda_S, \nu \vert_{\Lambda_S})$ and $L_p([0,1], \mathcal{A}, \lambda)$ are isometrically isomorphic as Banach spaces, where $1 \leq p < \infty$. On the other hand, if $S$ is a finite subset of $\mathbb{N}$, then $L_p(\lbrace 0,1 \rbrace^{\mathbb{N}}, \Lambda_S, \nu \vert_{\Lambda_S})$ is a finite-dimensional Banach space; this follows easily from the fact that $\Lambda_S$ is a finite set in that case.
\end{Rem}

The above technique is well known among experts in Banach space theory. We refer the interested reader to \cite{jsch} for a more sophisticated approach. 

Part (2) of the following result we learned from William~B.~Johnson, and it forms part of ongoing joint work between W.~B.~Johnson, N.~C.~Phillips and G.~Schechtman. With their kind permission we give our version of the proof here.

\begin{Prop}\label{exidemp}
	Let $X$ be a Banach space such that one of the following two conditions is satisfied.
	\begin{enumerate}[label=\textup{(\arabic*)}]
		\item $X$ has a subsymmetric Schauder basis; or
		\item $X = L_p[0,1]$ where $1 \leq p < \infty$.
	\end{enumerate}
	Then $\mathcal{B}(X)$ admits a bounded set $\mathcal{Q}$ of commuting, almost orthogonal idempotents such that $\vert \mathcal{Q} \vert = 2^{\aleph_0}$ and $\Ran(P) \simeq X$ for every $P \in \mathcal{Q}$. 
\end{Prop}

\begin{proof}
	We take an almost disjoint family $\mathcal{D}$ of continuum cardinality consisting of infinite subsets of $\mathbb{N}$.
	
	(1) Suppose $X$ has a subsymmetric Schauder basis $(b_n)$ with coordinate functionals $(f_n)$. Let $\mathcal{Q}= \lbrace P_N \rbrace_{N \in \mathcal{D}}$, where for $N \in \mathcal{D}$
\begin{align*}
P_N x= \sum\limits_{n \in N} \langle x, f_n \rangle b_n \quad (\text{for all } x \in X)
\end{align*}
defines an idempotent in $\mathcal{B}(X)$. Clearly $P_N P_M = P_{N \cap M} = P_M P_N$ has finite rank for distinct $N, M \in \mathcal{D}$. Also, $\Ran(P_N) \simeq X$ for every $N\in\mc D$ due to Proposition~\ref{subsymnorm}. The set $\mathcal{Q}$ is bounded by the second paragraph in Section~\ref{bases}.
	
	(2) In the notation of Remark~\ref{condexp}, for every $N \in \mathcal{D}$ we consider the conditional expectation operator
	\begin{align}
	\mathbb{E}(\cdot \vert \Lambda_N) \colon L_p(\lbrace 0, 1 \rbrace^{\mathbb{N}}, \Lambda, \mu) \rightarrow L_p(\lbrace 0, 1 \rbrace^{\mathbb{N}}, \Lambda, \mu); \quad f \mapsto \mathbb{E}(f \vert \Lambda_N).
	\end{align}
	By \cite{albiackalton}*{Lemma~6.1.1}, for any $N \in \mathcal{D}$ the bounded linear operator $\mathbb{E}(\cdot \vert \Lambda_N)$ is a norm one idempotent with range $L_p(\lbrace 0, 1 \rbrace^{\mathbb{N}}, \Lambda_N, \mu \vert_{\Lambda_N})$. In particular $\Ran(\mathbb{E}(\cdot \vert \Lambda_N))$ is isomorphic to $L_p([0,1], \mathcal{A}, \lambda)$ for each $N \in \mathcal{D}$. It follows from \cite{Fremlin}*{Theorem~254Ra} that for any two distinct $N,M \in \mathcal{D}$
	\begin{align*}
	\mathbb{E}(\cdot \vert \Lambda_N) \mathbb{E}(\cdot \vert \Lambda_M) = \mathbb{E}(\cdot \vert \Lambda_{N \cap M}),
	\end{align*}
	where $\Ran \big( \mathbb{E}(\cdot \vert \Lambda_{N \cap M}) \big) =L_p(\lbrace 0, 1 \rbrace^{\mathbb{N}}, \Lambda_{N \cap M}, \mu \vert_{\Lambda_{N \cap M}})$ is finite-dimensional.
	
	Let $T \colon L_p([0,1], \mathcal{A}, \lambda) \rightarrow L_p(\lbrace 0, 1 \rbrace^{\mathbb{N}}, \Lambda, \mu)$ be an isomorphism. Let $P_N= \mathbb{E}(\cdot \vert \Lambda_N)$ and $Q_N= T^{-1}   P_N   T$ for all $N \in \mathcal{D}$. Then $Q_N \in \mathcal{B}(L_p[0,1])$ is idempotent with $\Ran(Q_N) \simeq \Ran(P_N)$ and thus 
\begin{align*}
\Ran(Q_N) \simeq \Ran(P_N) = L_p(\lbrace 0, 1 \rbrace^{\mathbb{N}}, \Lambda_N, \mu \vert_{\Lambda_N}) \simeq L_p([0,1], \mathcal{A}, \lambda).
\end{align*}
Since $\Ran(Q_N Q_M)$ is finite-dimensional for distinct $N,M \in \mathcal{D}$ we obtain that the set $\mathcal{Q}= \lbrace Q_N \rbrace_{N \in \mathcal{D}}$ satisfies all of our requirements.
\end{proof}

The following fact is standard; we leave its proof to the reader.

\begin{Lem}\label{distofidemps}
Let $X$ be a Banach space. Let $\{Q_i\}_{i \in I}$ be a bounded set of mutually orthogonal, non-zero idempotents in $\mathcal{B}(X)$. Then for the density of $X$ we have $\density(X) \geq | I |$.
\end{Lem}

\begin{Rem}\label{idempinid}
Let $\mc A$ be an algebra, let $\mc J \trianglelefteq \mc A$ be a two-sided ideal. If $p,q \in \mc A$ are idempotents with $p \sim q$, then $p \in \mc J$ if and only if $q \in \mc J$. Indeed, let $a,b \in \mathcal{A}$ be such that $ab=p$ and $ba=q$. Hence $p=p^2=abab=aqb$ and similarly $q=bpa$.
\end{Rem}

The following proposition is a dichotomy result about separable representations of $\mc B (X)$ for certain Banach spaces $X$, in the sense of Berkson and Porta \cite{berksonporta}. In particular, Proposition~\ref{injorzero} generalizes their result \cite{berksonporta}*{Corollary~6.16}.

\begin{Prop}\label{injorzero}
Let $X$ be a Banach space such that one of the following two conditions is satisfied.
	\begin{enumerate}[label=\textup{(\arabic*)}]
		\item $X$ has a subsymmetric Schauder basis; or
		\item $X = L_p[0,1]$ where $1 \leq p < \infty$.
	\end{enumerate}
Let $Y$ be a separable Banach space. Let $\theta \colon \mc B (X) \rightarrow \mc B (Y)$ be a continuous algebra homomorphism. Then $\theta$ is either injective or $\theta =0$.
\end{Prop}

\begin{proof}
Assume towards a contradiction that $\theta$ is not injective and $\theta \neq 0$. In particular $\mc K (X) \subseteq \Ker(\theta)$. By Proposition~\ref{exidemp}, $\mathcal{B}(X)$ admits a bounded set $\{P_i\}_{i \in I}$ of commuting, almost orthogonal idempotents such that $| I | = 2^{\aleph_0}$ and $\Ran(P_i) \simeq X$ for each $i \in I$. We \textit{claim} that $\{\theta(P_i)\}_{i \in I}$ is a bounded set of mutually orthogonal, non-zero idempotents of continuum cardinality in $\mc B (Y)$. To see this, we observe first that  $\theta(P_i) \theta(P_j) = \theta(P_i P_j) =0$ for each distinct $i,j \in I$, as $P_iP_j \in \mc K (X) \subseteq \Ker(\theta)$. Now observe that $\theta(P_i)$ is non-zero for each $i \in I$. Indeed, $I_X \notin \Ker(\theta)$ as $\theta$ is non-zero. Since $\Ran(P_i) \simeq X$, in view of Lemma~\ref{equivisom} this means $P_i \sim I_X$ for all $i \in I$. Thus by Remark~\ref{idempinid} we obtain $P_i \notin \Ker(\theta)$ for all $i \in I$. This shows the claim. But now with Lemma~\ref{distofidemps} we obtain $\density(Y) \geq 2^{\aleph_0}$, a contradiction.
\end{proof}

\section{Proof of the main results}

\subsection{The proof of Theorem \ref{main}}

We are now in position to prove our main result. Before we get to it, let us mention that the techniques below are akin to those employed by Moln\'ar to Hilbert spaces in \cite{molnar1}. Some of these techniques go back to at least Johnson's seminal work on approximately multiplicative maps between Banach algebras in \cite{amnm2}.

\begin{proof}[Proof of Theorem~\ref{main}]
Since $Y$ is reflexive, from Lemma~\ref{bxisdualbanalg} we know that $\mc B (Y)$ is a dual Banach algebra with predual $Y \widehat{\otimes}_{\pi} Y^*$.
	
If $X$ has a subsymmetric basis, let this be denoted by $(b_n)$. If $X = L_p[0,1]$, where $1<p< \infty$, then $(b_n)$ denotes the Haar basis. In both cases $(f_n)$ stands for the sequence of coordinate functionals associated to $(b_n)$. As $X$ is reflexive, it follows from Corollary~\ref{aiai2} and the comment after Proposition~\ref{subsymnorm} that the sequence of coordinate projections $(P_n)$ is a b.a.i. for $\mathcal{K}(X)$.
	
Since $(\phi(P_n))$ is a bounded, monotone increasing sequence of idempotents in $\mc B (Y)$ it follows from Lemma~\ref{projsot} that there exists an idempotent $P \in \mathcal{B}(Y)$ such that $(\phi(P_n))$ converges to $P$ in the strong operator topology. We show that in fact $P = I_Y$. To this end we consider the map
\begin{align*}
\theta \colon \mc B(X) \rightarrow \mc B (Y); \quad A \mapsto (I_Y - P) \phi(A) (I_Y -P).
\end{align*}
By Lemma~\ref{techlem1} the map $\theta$ is a continuous algebra homomorphism with $\theta |_{\mc K (X)} =0$. Due to separability of $Y$, Proposition~\ref{injorzero} yields $\theta =0$. By the assumption, we can take $T \in \mc B (X)$ such that $\phi(T)$ has dense range. Consequently
\begin{align*}
0= \theta(T) = (I_Y - P) \phi(T) (I_Y - P) = (I_Y -P) \phi(T)
\end{align*}
by \eqref{denserange}. So $(I_Y -P) |_{\Ran(\phi(T))} =0$ and $\Ran(\phi(T))$ is dense in $Y$, hence $P=I_Y$.

Let $x_0 \in X$ be such that $\|x_0\| = 1$, and choose $f_0 \in X^*$  such that $\langle x_0, f_0 \rangle =1 = \|f_0 \|$. As $\phi$ is injective, we can pick $y_0 \in Y^*$ with $\| y_0 \| =1$ such that $\phi(x_0 \otimes f_0) y_0 \neq 0$. Thus we can define the non-zero map
\begin{align*}
S \colon X \rightarrow Y; \quad x \mapsto \phi(x \otimes f_0)y_0
\end{align*}
which is easily seen to be linear and bounded. We observe that
\begin{align}\label{implem}
SA = \phi(A)   S \qquad (\text{for all } A \in \mc B (X)).
\end{align}
Indeed, fix $A \in \mc B (X)$ and $x \in X$. Then
\begin{align*}
\phi(A)Sx = \phi(A)\phi(x \otimes f_0)y_0 = \phi(A(x \otimes f_0))y_0 = \phi(Ax \otimes f_0)y_0 = SAx.
\end{align*}
In the following we show that $S$ is an isomorphism.

We observe that $S$ is injective. For assume in search of a contradiction it is not; let $x \in X$ satisfy $Sx = 0$ and $\|x \|=1$. Let $f \in X^*$ be such that $\langle x, f \rangle = 1 = \|f \|$. Then in view of \eqref{implem} we have
\begin{align*}
0= \phi(z \otimes f)Sx = S(z \otimes f)x = S(\langle x, f \rangle z) = Sz \quad (\text{for all }z \in X).
\end{align*}
Thus $S=0$, a contradiction.

We show that $S$ has closed range. To this end, let $(x_n)$ be a sequence in $X$ such that $(Sx_n)$ converges to some $y \in Y$. Let $x \in X$ be non-zero. As $S$ is injective, we have $Sx \neq 0$. Thus we can choose $h \in Y^*$ with $\langle Sx, h \rangle =1$. Let $f \in X^*$ be arbitrary fixed, then by \eqref{implem} we have
\begin{align*}
\langle x_n, f \rangle Sx = S(x \otimes f)x_n =\phi(x \otimes f)Sx_n \quad (\text{for all }n \in \mathbb{N}).
\end{align*}
Hence $\langle x_n, f \rangle Sx\to\phi(x \otimes f)y \in Y$ and therefore $\langle x_n, f \rangle\to \langle \phi(x \otimes f)y, h \rangle$. As $f \in X^*$ was arbitrary, this shows that $(x_n)$ is a weak Cauchy sequence in $X$. Since $X$ is reflexive, it is weakly sequentially complete (see \textit{e.g.} \cite{conway}*{Chapter~V, Corollary~4.4}), hence $(x_n)$ converges weakly to some $x' \in X$. As $S$ is weakly continuous, $Sx_n\to Sx'$ weakly in $Y$. But $(Sx_n)$ converges in norm to $y \in Y$, so it also converges to $y$ weakly. By uniqueness of the weak limit $Sx' =y$.

It remains to show that that $S$ has dense range. Clearly $b_n \otimes f_n \in \mc B (X)$ is a rank one idempotent, hence by the assumption $\phi(b_n \otimes f_n) \in \mc B (Y)$ is a rank one idempotent too for each $n \in \mathbb{N}$. Let $u_n \in Y$ and $h_n \in Y^*$ be such that $\phi(b_n \otimes f_n)= u_n \otimes h_n$ and $\langle u_n, h_n \rangle =1$. Recall that $(\phi(P_n))$ converges to $I_Y$ in the strong operator topology, consequently
\begin{align*}
x= \lim\limits_{n \rightarrow \infty} \phi(P_n)x = \sum\limits_{i=1}^{\infty} \phi(b_i \otimes f_i)x = \sum\limits_{i=1}^{\infty} (u_i \otimes h_i)x = \sum\limits_{i=1}^{\infty} \langle x, h_i \rangle u_i \quad (\text{for all }x \in X).
\end{align*}
This shows $X= \overline{\spanning} \{u_n : \, n \in \mathbb{N} \}$. To conclude the proof, it suffices to show that $u_n \in \Ran(S)$ for each $n \in \mathbb{N}$. This essentially follows from \eqref{implem}, as 
\begin{align}\label{mainfinal}
Sb_n = S(b_n \otimes f_n)b_n = \phi(b_n \otimes f_n)Sb_n = (u_n \otimes h_n)Sb_n = \langle Sb_n, h_n \rangle u_n
\end{align}
for each $n \in \mathbb{N}$. Injectivity of $S$ implies that $Sb_n$ is non-zero, hence $\langle Sb_n, h_n \rangle \neq 0$ by \eqref{mainfinal}. Therefore $u_n \in \Ran(S)$ indeed. 

Thus \eqref{implem} amounts to
\begin{align*}
\phi(A) = S   A   S^{-1} \quad (\text{for all }A \in \mc B (X)),
\end{align*}
which proves that $\phi$ is an isomorphism.
\end{proof}

\begin{Ex}
Each of the following spaces is reflexive and has a subsymmetric basis, hence satisfies the conditions of Theorem~\ref{main}~(2):
\begin{enumerate}
	\item[(a)] The sequence spaces $\ell_p$, where $1<p<\infty$ (see Section \ref{bases});
	\item[(b)] every reflexive Orlicz sequence space $l_M$ with Orlicz function $M$ satisfying the $\Delta_2$-condition $\limsup_{t\to0} M(2t)/M(t)<\infty$ (see \cite{lindenstrauss1}*{Propositions~4.a.4 and 3.a.3});
	\item[(c)] every Lorentz sequence space $d(w,p)$, where $p>1$, $w=(w_n)_{n\in\mathbb{N}}$ is non-increasing, $w_1=1$, $\lim_{n\to\infty} w_n=0$ and $\sum_{n=1}^\infty w_n=\infty$ (see \cite{lindenstrauss1}*{Propositions~4.e.3~and~1.c.12}).
\end{enumerate}
\end{Ex}

\begin{Rem}
	In the proof of Theorem~\ref{main} the cornerstone of our argument is that $X$ has \textit{uncountably many} complemented subspaces, each of which is isomorphic to $X$ itself, but any two have a finite-dimensional intersection. We hope that Remark~\ref{conclrem} at the end of the paper sheds some light on why this phenomenon might be essential.
\end{Rem}

\subsection{The proof of Theorem~\ref{stability}}

In the following let $X$ and $Y$ be arbitrary non-zero Banach spaces, and let $\psi, \phi \colon \mc B (X) \rightarrow \mc B (Y)$ be algebra homomorphisms such that $\| \psi(A) - \phi(A) \| < \|A\|$ for each non-zero $A \in \mc B (X)$. 

With an application of the triangle inequality we arrive at the simple but useful estimate
\begin{align}\label{strictineq}
\|\psi(A)\| &\leq \| \psi(A) - \phi(A) \| + \|\phi(A) \| < \| A \| + \| \phi(A) \|.
\end{align}
Similarly we obtain $\|\phi(A)\| < \| A \| + \| \psi(A) \|$. In particular, these estimates immediately yield that $\phi$ is continuous if and only if $\psi$ is continuous.

\begin{Lem}\label{injeq}
Let $P \in \mc B (X)$ be a norm one idempotent. Then $P \in \Ker(\phi)$ if and only if $P \in \Ker(\psi)$. Consequently, $\psi$ is injective if and only if $\phi$ is injective.
\end{Lem}

\begin{proof}
Assume $P \in \Ker(\phi)$. Then it follows from \eqref{strictineq} that $\|\psi(P)\| < \|P \| = 1$. As $\psi(P) \in \mc B (Y)$ is an idempotent, this is equivalent to saying $\psi(P) =0$. The other direction follows analogously.

In order to show the ``consequently'' part suppose contrapositively that $\psi$ is not injective. Let $x \in X$ be such that $\|x \| =1$. Pick an $f \in X^*$ with $\langle x,f \rangle = 1 = \|f\|$. Hence $x \otimes f \in \mc F (X)$ is a norm one idempotent. In particular $x \otimes f \in \Ker(\psi)$, which by the first part of the lemma is equivalent to $x \otimes f \in \Ker(\phi)$. This shows that $\phi$ is not injective. Similarly, one obtains that injectivity of $\psi$ implies injectivity of $\phi$.
\end{proof}

\begin{Prop}\label{corzemanek}
Let $P \in \mc B (X)$ be a norm one idempotent. Then $\Ran(\psi(P)) \simeq \Ran(\phi(P))$. If $\psi$ is surjective, then $\phi(I_X) = I_Y$. Moreover, if $\psi$ is an isomorphism, then $\Ran(\phi(P)) \simeq \Ran(P)$.
\end{Prop}

\begin{proof}
As $\|P\| =1$, the estimate $\| \psi(P) - \phi(P) \| <1$ and Lemma~\ref{zemanek} imply $\psi(P) \sim \phi(P)$. In view of Lemma~\ref{equivisom} this is equivalent to saying $\Ran(\psi(P)) \simeq \Ran(\phi(P))$.

Suppose $\psi$ is surjective, then $\psi(I_X) = I_Y$. Indeed, this immediately follows as $\psi(I_X)A = A = A \psi(I_X)$ for each $A \in \mc B(Y)$. Therefore
\begin{align*}
\| I_Y - \phi(I_X) \| = \| \psi(I_X) - \phi(I_X) \| < 1,
\end{align*}
which by the Carl Neumann series implies that $\phi(I_X)$ is invertible in $\mc B (Y)$. As $\phi(I_X)$ is an idempotent, $\phi(I_X) = I_Y$ must hold.

Suppose $\psi$ is an isomorphism. By Eidelheit's Theorem (see \cite{eid}*{Theorem~2}) there is an isomorphism $S \in \mc B (X,Y)$ such that $\psi(A)= S   A   S^{-1}$ for each $A \in \mc B (X)$. In particular, $(S   P)   (P   S^{-1}) = S   P   S^{-1} = \psi(P)$ and $(P   S^{-1})   (S   P) =P$ imply (with Lemma~\ref{equivisom}) that $\Ran(P) \simeq \Ran(\psi(P))$. By the first part of the proposition $\Ran(\phi(P)) \simeq \Ran(P)$ follows.
\end{proof}

From this point on, we assume that the properties prescribed by the conditions of Theorem~\ref{main} hold for the Banach spaces $X$ and $Y$, and $\psi \colon \mc B (X) \rightarrow \mc B (Y)$ is assumed to be surjective. We recall that due to the deep automatic continuity result of B.~E.~Johnson \cite{johnson1967uniqueness}, any surjective homomorphism between algebras of operators of Banach spaces is automatically continuous (see \emph{e.g.} \cite{Dales}*{Theorem~5.1.5} for a detailed proof).

Outfitted with Theorem~\ref{main} and the results above, we can now prove Theorem~\ref{stability}.

\begin{proof}[Proof of Theorem~\ref{stability}]
We first observe that $\psi$ is automatically injective. Indeed, $Y$ is non-zero, hence $\psi$ is non-zero, since it is surjective. By Proposition~\ref{injorzero} it follows that $\psi$ is injective.

Thus by Lemma~\ref{injeq},  $\phi$ is injective too. Continuity of $\psi$ and \eqref{strictineq} imply that $\phi$ is continuous. Furthermore, from Proposition~\ref{corzemanek} we conclude that $\phi(I_X)= I_Y$ (which witnesses that $\Ran(\phi)$ contains an operator with dense range), and $\phi$ preserves rank one idempotents. Hence Theorem~\ref{main} applies.
\end{proof}

\subsection{The proof of Proposition~\ref{counterexample}}

In each of the following examples, $\mathcal{B}(X)$ has a character. In examples (1)--(3) this character is shown explicitly and in example (4) the character is obtained from a commutative quotient of $\mathcal{B}(X)$. We remark in passing that the list below is not intended to be comprehensive.

\begin{Ex}\label{nonex}
	Each of the following spaces $X$ is such that $\mc B (X)$ has a character:
	\begin{enumerate}
		\item[(1)] $X= \mathcal{J}_p$ where $1<p< \infty$ and $\mathcal{J}_p$ is the $p^{th}$ James space, since by \cite{edmit}*{Paragraph~8} $\mathcal{B}(X)$ has a character whose kernel is $\mathcal{W}(X)$, the \textit{ideal of weakly compact operators} (see also \cite{laustsenmax1}*{Theorem~4.16});
		\item[(2)] $X=C[0, \omega_1]$, where $\omega_1$ is the first uncountable ordinal, since by \cite{edmit}*{Paragraph~9} $\mathcal{B}(X)$ has a character (see also \cite{lw}*{Proposition~3.1});
		\item[(3)] $X=X_{GM}$ is the \textit{hereditarily indecomposable} Banach space constructed by Gowers and Maurey in \cite{GM0}, since $\mathcal{B}(X)$ has a character whose kernel is $\mathcal{S}(X)$, the ideal of \textit{strictly singular operators};
		\item[(4)] $X= \mathcal{G}$, where $\mathcal{G}$ is the Banach space constructed by Gowers in \cite{gowers1}, since $\mathcal{B}(X) / \mathcal{S}(X) \simeq \ell_{\infty} / c_0$, as shown in \cite{laustsenmax1}*{Corollary~8.3}.
	\end{enumerate}
\end{Ex}

\begin{proof}[Proof of Proposition~\ref{counterexample}]
Let $\chi \colon \mc B (X) \rightarrow \mathbb{C}$ be a character. Let $Z$ be a non-zero Banach space, and consider the map
\begin{align*}
 \phi \colon \mc B (X) \rightarrow \mc B (X \oplus Z); \quad T \mapsto
 \begin{bmatrix}
 T & 0\\
 0 & \chi(T) I_Z
 \end{bmatrix}. 
\end{align*}
From $\chi(I_X) = 1$ it is immediate that $\phi(I_X) = I_{X \oplus Z}$. As $\chi$ is a norm one algebra homomorphism, it readily follows that $\phi$ is a norm one algebra homomorphism too. The map $\phi$ is clearly injective. Let $x_0 \in X$ and $f_0 \in X^*$ satisfy $\langle x_0, f_0 \rangle \neq 0$. As $\chi$ is a character of $\mc B (X)$ and $\mc F (X)$ is the smallest non-trivial, two-sided ideal of $\mc B(X)$, we have $x_0 \otimes f_0 \in \mc F (X) \subseteq \Ker (\chi)$. Thus
\begin{align*}
\phi(x_0 \otimes f_0) =
\begin{bmatrix}
x_0 \otimes f_0 & 0\\
0 & 0
\end{bmatrix} =
\begin{bmatrix}
x_0 \\
0
\end{bmatrix} \otimes 
\begin{bmatrix}
f_0 \\
0
\end{bmatrix},
\end{align*}
from which it also follows that $\phi$ maps rank one operators into rank one operators. Finally, it is obvious that $\phi$ cannot be surjective.

The second part of the proposition is an immediate corollary of Examples~\ref{nonex} (1)--(2), the first part of the proposition with the choice $Z= \mathbb{C}$, and the fact that $X \simeq X \oplus \mathbb{C}$. Although the latter is certainly well-known, for completeness we give the details:

(1) \textit{Let $X= \mc{J}_p$, where $1<p< \infty$.} Recall that the James space is one-codimensional in its bidual and it is isometrically isomorphic to its bidual (see \textit{e.g.} \cite{albiackalton}*{Theorem~3.4.6}). Consequently $X \simeq X^{**} \simeq X \oplus \mathbb{C}$.

(2) \textit{Let $X= C[0, \omega_1]$.} As $X$ has a complemented copy of $c_0$ (see \cite{acgjm}*{Proposition~3.2}), and of course $c_0 \simeq c_0 \oplus \mathbb{C}$, we conclude $X \simeq X \oplus \mathbb{C}$.
\end{proof}

\begin{Rem}\label{conclrem}
	In light of Proposition~\ref{counterexample} and Example~\ref{nonex} let us make a few remarks about possible weakenings of the conditions in Theorem~\ref{main}. In the following $Y= X \oplus \mathbb{C}$.
\begin{itemize}
\item Let $X= X_{GM}$. It is shown \cite{GM0} that $X_{GM}$ is reflexive and has a Schauder basis, and hence $Y$ is separable and reflexive. This shows that in Theorem~\ref{main}, the conditions on $X$ cannot be weakened to ``$X$ is reflexive and has a Schauder basis''. 	
\item Let $X= \mc G$. It is shown in \cite{gowers1} that $\mc G$ has an unconditional Schauder basis, hence $Y$ is separable. This shows that in Theorem~\ref{main}, the conditions on $X$ and $Y$ cannot be weakened to ``$X$ has an unconditional basis and $Y$ is separable''.
\end{itemize}
\end{Rem}

\begin{Ack}
The authors are grateful to W.~B.~Johnson (Texas A$\&$M) for the useful e-mail exchanges regarding Proposition~\ref{exidemp}. They are also indebted to N.~J.~Laustsen (Lancaster) for his invaluable suggestions and for carefully reading the first draft of the manuscript. We are extremely grateful to the anonymous referee whose detailed suggestions and comments much improved the exposition of the paper.
\end{Ack}


\bibliographystyle{amsplain}


\Addresses

\end{document}